\documentclass[oneside,english,a4wide]{amsart}
\usepackage[latin9]{inputenc}
\usepackage[a4paper]{geometry}
\usepackage{mathabx}
\usepackage{esvect}
\usepackage{babel}
\usepackage{mathtools}
\usepackage{amstext}
\usepackage{amsthm}
\usepackage{amssymb}
\usepackage{enumerate}
\usepackage{esint}
\usepackage{graphicx}
\usepackage{bbm}
\usepackage[all]{xy}
\usepackage{mathrsfs}
\usepackage[unicode=true,
bookmarks=true,bookmarksnumbered=true,bookmarksopen=true,bookmarksopenlevel=2,
breaklinks=false,pdfborder={0 0 1},backref=false,colorlinks=false]
{hyperref}
\hypersetup{
	colorlinks,linkcolor=blue,anchorcolor=blue,citecolor=blue}
\usepackage{breakurl}
\usepackage{autobreak}

\usepackage[svgnames]{xcolor} 

\usepackage{graphicx} 

\usepackage{amsmath}

\makeatletter
\numberwithin{equation}{section}
\numberwithin{figure}{section}
\theoremstyle{plain}
\newtheorem{thm}{\protect\theoremname}[section]
\theoremstyle{plain}

\newtheorem{proposition}[thm]{Proposition}

\theoremstyle{definition}

\theoremstyle{plain}
\newtheorem{lem}[thm]{\protect\lemmaname}
\newtheorem{lemma}[thm]{Lemma}
\theoremstyle{remark}
\newtheorem{rem}[thm]{\protect\remarkname}

\theoremstyle{definition}
\newtheorem{definition}[thm]{Definition}
\newtheorem*{claim*}{Claim}

\theoremstyle{remark}

\theoremstyle{definition}
\newtheorem*{defn*}{\protect\definitionname}

\usepackage{babel}
\usepackage{amscd}

\makeatother

\providecommand{\definitionname}{Definition}
\providecommand{\lemmaname}{Lemma}
\providecommand{\propositionname}{Proposition}
\providecommand{\remarkname}{Remark}
\providecommand{\theoremname}{Theorem}

\newcommand{\Rmnum}[1]{\expandafter\@slowromancap\romannumeral #1@}
\makeatother

\newcommand{\M}{{\mathcal M}}

\newcommand{\8}{\infty}

\newcommand{\la}{\langle}
\newcommand{\ra}{\rangle}

\newcommand{\be}{\begin{eqnarray*}}
	\newcommand{\ee}{\end{eqnarray*}}
\newcommand{\beq}{\begin{equation}}
	\newcommand{\eeq}{\end{equation}}
\newcommand{\beqn}{\begin{equation*}}
	\newcommand{\eeqn}{\end{equation*}}
\newcommand{\bs}{\begin{split}}
	\newcommand{\es}{\end{split}}

\begin{document}

	
	
	\title[Boundedness of operator-valued commutators]{Boundedness of operator-valued commutators involving martingale paraproducts}
	
	\thanks{{\it 2020 Mathematics Subject Classification:} Primary: 46L53, 60G42. Secondary: 46L52, 47B47}
	\thanks{{\it Key words:} Martingale paraproducts, semicommutative  martingales, commutators, $BMO$ spaces}
	
	\author[Zhenguo Wei]{Zhenguo Wei}
	\address{
		Laboratoire de Math{\'e}matiques, Universit{\'e} de Bourgogne Franche-Comt{\'e}, 25030 Besan\c{c}on Cedex, France}
	\email{zhenguo.wei@univ-fcomte.fr}

	\author[Hao Zhang]{Hao Zhang}
	\address{
		Department of Mathematics, University of Illinois Urbana-Champaign, USA}
	\email{hzhang06@illinois.edu}
	\date{}
	\maketitle
	
	\begin{abstract}
      Let $1<p<\infty$. We show the boundedness of operator-valued commutators $[\pi_a,M_b]$ on the noncommutative $L_p(L_\infty(\mathbb{R})\otimes \mathcal{M})$ for any von Neumann algebra $\mathcal{M}$, where $\pi_a$ is the $d$-adic martingale paraproduct with symbol $a\in BMO^d(\mathbb{R})$ and $M_b$ is the noncommutative left multiplication operator with $b\in BMO^d_\mathcal{M}(\mathbb{R})$. Besides, we consider the extrapolation property of semicommutative $d$-adic martingale paraproducts in terms of the $BMO^d_\mathcal{M}(\mathbb{R})$ space.
	\end{abstract}
	
	
	\section{Introduction}\label{Introduction}

Martingale paraproducts have become an important model in modern harmonic analysis. They have been successfully proved to be a fundamental and useful tool to study a variety of properties for singular integral operators and commutators. For example, Coifman and Semmes used dyadic martingale paraproducts to show the celebrated David-Journé $T1$ theorem. We refer the reader to \cite{CJS}, \cite{DaJou} and \cite{Per} for more details about the $T1$ theorem and its dyadic proof. As an extension of the work of David-Journé and Coifman-Semmes, Nazarov, Treil and Volberg applied dyadic martingale paraproducts to the $T1$ and $Tb$ theorem on non-homogeneous spaces in \cite{NTV}.
Refining the method of Nazarov, Treil and Volberg \cite{NTV}, Hyt\"{o}nen used dyadic martingale paraproducts  and dyadic shift operators to give a new dyadic representation for general singular integral operators to finally settle the famous $A_2$ conjecture in \cite{TH1}.
 We would like to point out that in \cite{Bony} Bony gave the systematic study of paraproducts in the context of the theory of paradifferential operators.

Apart from the close connection between martingale paraproducts and harmonic analysis, it is of independent interest to study martingale paraproducts in its own right. These operators are generalizations of Hankel type operators, an important class of operators in function theory. The boundedness of Hankel operators has been investigated by Nehari \cite{Neha} in terms of the $BMO$ space. We refer the reader to \cite{Hankel} and \cite{PJ} for more details about Hankel operators.

Let $\mathcal{M}$ be a von Neumann algebra equipped with a normal semifinite faithful trace $\tau$. Denote by $\pi_b$ the dyadic martingale paraproduct with symbol $b\in L^{\rm{loc}}_1({\mathbb{R}}, L_1(\M))$. Recently, the operator-valued $T1$ problem has attracted wide attention in noncommutative harmonic analysis. This is closely related to the boundedness of semicommutative martingale paraproducts in view of Hyt\"{o}nen's dyadic representation \cite{TH1}. The reader is referred to \cite{HM} for more details. However, Mei has shown in \cite{Mei} that in general,  $\|\pi_b\|_{B(L_2(L_\infty(\mathbb{R})\otimes \mathcal{M}))}$ cannot even be dominated by the operator norm $\|b\|_{L_\infty(\mathbb{R})\otimes \mathcal{M}}$ for infinite-dimensional $\M$ .

Hence, it is extremely difficult to investigate the boundedness of $\pi_b$ on $L_2(L_\infty(\mathbb{R})\otimes \mathcal{M})$. Indeed, when $\M=\mathbb{M}_{n}$ is the algebra of all $n\times n$ matrices, Katz employed an ingenious stopping time procedure in \cite{KNH} to show
\begin{equation}\label{Katz}
\|\pi_b\|_{B(L_2(L_\infty(\mathbb{R})\otimes \mathbb{M}_{n}))}\lesssim \log (n+1) \|b\|_{BMO_{so}^2(\mathbb{R}, \mathbb{M}_{n})},
\end{equation}
where 
\begin{equation*}
BMO_{so}^d(\mathbb{R}, \mathbb{M}_{n})=\sup_{x\in \mathbb{C}^n\atop \|x\|\le1} \sup_{I\in \mathcal{D}}\bigg(\frac{1}{m(I)}\int_I \Bigl\|\Big(b-\big(\frac{1}{m(I)} \int_I b\ d m\big)\Big)x\Bigr\|_{\mathbb{C}^n}^2 dm\biggr)^{1/2}<\infty.
\end{equation*}
We denote by $m$ Lebesgue measure, and $\mathcal{D}$ the family of all $d$-adic intervals on $\mathbb{R}$. We refer the reader to \cite{Mei1} for more information for such strong operator $BMO$ spaces.

Nazarov, Treil and Volberg also independently obtained  (\ref{Katz}) in \cite{NTV1} by the Bellman method, and they also gave an example showing that for any $n\in \mathbb{N}$ there exists $b$ such that
$$  \|\pi_b\|_{B(L_2(L_\infty(\mathbb{R})\otimes \mathbb{M}_{n}))}\gtrsim \sqrt{\log (n+1)} \|b\|_{BMO_{so}^2(\mathbb{R}, \mathbb{M}_{n})}.  $$
This implies that the boundedness of $\pi_b$ cannot be characterized solely by $BMO_{so}^2(\mathbb{R}, \mathcal{M})$ for infinite-dimensional $\M$. In \cite{NPTV}, Nazarov, Pisier, Treil and Volberg proved that $\log{(n+1)}$ is the optimal order of the best constant in (\ref{Katz}). It still remains open how to characterize the boundedness of $\pi_b$ on $L_2(L_\infty(\mathbb{R})\otimes \mathcal{M})$.

Although we do not know how to describe the boundedness of semicommutative martingale paraproducts until now, some particular properties of $\pi_b$ can be investigated. For instance, Mei succeeded in describing the extrapolation property by virtue of  $BMO^2_{\mathcal{M}}(\mathbb{R})$ for dyadic martingale paraproducts  in \cite{Mei2}. Recall that $BMO^d_{\mathcal{M}}(\mathbb{R})$ is the operator-valued $BMO$ space associated with the $d$-adic martingales consisting of all $\mathcal{M}$-valued functions $b$ that are Bochner integrable on any $d$-adic interval such that
\begin{equation}\label{BMOMd}
\|b\|_{BMO^d_\mathcal{M}(\mathbb{R})}=\sup_{I\in \mathcal{D}}\biggl(\frac{1}{m(I)}\int_I \Bigl\|b-\big(\frac{1}{m(I)} \int_I b\ d m\big)\Bigr\|_\mathcal{M}^2 dm\biggr)^{1/2}<\infty.
\end{equation} 
When $\mathcal{M}=\mathbb{C}$, we denote it by $BMO^d(\mathbb{R})$ for simplicity. This coincides with the martingale $BMO$ spaces associated with $d$-adic martingales in the commutative setting.


In \cite{WZ2024}, we use Hyt\"{o}nen's dyadic representation \cite{TH1} to prove the boundedness of the operator-valued commutator $[T, M_b]=T M_b - M_b T$ on $L_2(L_\infty(\mathbb{R})\otimes \mathcal{M})$, where $T$ is a general bounded singular integral operator on $L_2(\mathbb{R}^n)$ and $M_b$ is the left multiplication by $b$. During our proof, the operator-valued commutator $[\pi_a, M_b]$ naturally emerges for $a\in BMO(\mathbb{R}^n)$, where $[\pi_a, M_b]$ is given by
$$  [\pi_a, M_b](f)=\pi_a(M_b(f))-M_b(\pi_a(f))=\pi_a(b\cdot f)-b\cdot \pi_a(f) \quad \forall f\in  L_2(L_\infty(\mathbb{R})\otimes \mathcal{M}). $$
They establish the boundedness of the commutator $[\pi_a, M_b]$ on $L_2(L_\infty(\mathbb{R})\otimes \mathcal{M})$.

 Motivated by the aforementioned work, this article is devoted to the boundedness of operator-valued commutators $[\pi_a, M_b]$ on $L_p(L_\infty(\mathbb{R})\otimes \mathcal{M})$ for all $1<p<\infty$. At first, we introduce semicommutative $d$-adic martingale paraproducts. Let $d\geq 2$ be a natural number. Given a semicommutative $d$-adic martingale $b=(b_k)_{k\in\mathbb{Z}}\in L^{\rm{loc}}_1({\mathbb{R}}, L_1(\M))$, the martingale paraproduct with symbol $b$ is defined as follows: $\forall f=(f_k)_{k\in\mathbb{Z}}\in L_2(L_\infty(\mathbb{R})\otimes \mathcal{M})$
\begin{equation*}
	\pi_b(f)=\sum_{k=-\infty}^{\infty}d_kb \cdot f_{k-1},
\end{equation*}
where $d_kb=b_k-b_{k-1}$ for any $k\in\mathbb{Z}$. See Subsection \ref{sec2.3} for the definition of semicommutative $d$-adic martingales. When $d=2$, $d$-adic martingales are reduced to dyadic martingales. 


The following theorem concerns the boundedness of operator-valued commutators involving  $d$-adic martingale paraproducts.
\begin{thm}\label{3}
	Let $1<p<\infty$. If $a\in BMO^d(\mathbb{R})$ and $b\in BMO^d_\mathcal{M}(\mathbb{R})$, then $[\pi_a,M_b]$ is bounded on $L_p(L_\infty(\mathbb{R})\otimes \mathcal{M})$ and
	\begin{equation*}
	\|[\pi_a,M_b]\|_{L_p(L_\infty(\mathbb{R})\otimes \mathcal{M})\to L_p(L_\infty(\mathbb{R})\otimes \mathcal{M})}\lesssim_{d,p} \|a\|_{BMO^d(\mathbb{R})}\|b\|_{BMO^d_{\mathcal{M}}(\mathbb{R})}.
	\end{equation*}
\end{thm}

\begin{rem}
	When $\M=\mathbb{C}$, Theorem \ref{3} is shown in \cite{WZ2024}.
	\end{rem}

Assume $a\in BMO^d(\mathbb{R})$. It is clear that $\pi_a$ is bounded on $L_2(L_\infty(\mathbb{R})\otimes \mathcal{M})$ since $\pi_a$ is bounded on $L_2(\mathbb{R})$. To show Theorem \ref{3}, the boundedness of $\pi_a$ on $L_p(L_\infty(\mathbb{R})\otimes \mathcal{M})$ $(1<p<\8)$ will be needed. To this end, we need to generalize the extrapolation property of $\pi_b$ \cite{Mei2} to semicommutative $d$-adic martingales for any $d\geq 2$.
\begin{thm}\label{2}
	Let $1<p<\infty$ and $b\in BMO^d_\mathcal{M}(\mathbb{R})$. If $\pi_b$ is bounded on $L_p(L_\infty(\mathbb{R})\otimes \mathcal{M})$ for some $1<p<\infty$, then it is bounded on $L_p(L_\infty(\mathbb{R})\otimes \mathcal{M})$ for all $1<p<\infty$.
\end{thm}

Mei considered the extrapolation property only for the dyadic martingale, i.e. $d=2$ in \cite{Mei2}. We extend this property to general $d\geq 3$.


Another ingredient of the proof of Theorem \ref{3} is the bilinear decomposition of $M_b$. Assume $b\in L_2(L_\infty(\mathbb{R})\otimes \mathcal{M})$. Define for any $f\in L_2(L_\infty(\mathbb{R})\otimes \mathcal{M})$
\begin{equation}\label{rab}
	\varLambda_b(f)=\sum_{k\in\mathbb{Z}}d_kb\cdot d_kf \quad \text{and} \quad R_b(f)=\sum_{k\in\mathbb{Z}}b_{k-1}\cdot d_kf.
\end{equation}
Note that for $b, f\in L_2(L_\infty(\mathbb{R})\otimes \mathcal{M})$,  $M_b(f)=\pi_b(f)+\varLambda_b(f)+R_b(f)$.

As mentioned before, $b\in BMO^d_\mathcal{M}(\mathbb{R})$ does not guarantee the boundedness of $\pi_b$.  We need to consider the sum of $\pi_b$ and $\varLambda_b$. Hence, we  define
\begin{equation}\label{theta}
	\Theta_b=\pi_b+\varLambda_b.
\end{equation}
We refer to \cite{HM} for the Hardy space $h_{p,c}^d(\mathbb{R})$ of $d$-adic martingales.
\begin{thm}\label{1}
	Let $1<p<\infty$ and $b\in BMO^d_\mathcal{M}(\mathbb{R})$.
	
		$\mathrm{(1)}$ If $2\le p< \infty$, then $\Theta_b$ is bounded from $L_p(L_\infty(\mathbb{R})\otimes \mathcal{M})$ to $h_{p,c}^d(\mathbb{R})$ and
		\begin{equation*}
			\|\Theta_b\|_{L_p(L_\infty(\mathbb{R})\otimes \mathcal{M})\to h_{p,c}^d(\mathbb{R})}\lesssim_{d,p} \|b\|_{BMO^d_\mathcal{M}(\mathbb{R})}.
		\end{equation*}
		
		$\mathrm{(2)}$ If $1< p\le 2$, then $\Theta_b$ is bounded from $h_{p,c}^d(\mathbb{R})$ to $L_p(L_\infty(\mathbb{R})\otimes \mathcal{M})$ and 
		\begin{equation*}
			\|\Theta_b\|_{h_{p,c}^d(\mathbb{R})\to L_p(L_\infty(\mathbb{R})\otimes \mathcal{M})}\lesssim_{d,p} \|b\|_{BMO^d_\mathcal{M}(\mathbb{R})}.
		\end{equation*}
\end{thm}

It is shown that Theorem \ref{1} holds when $p=2$ in \cite{WZ2024}. Theorem \ref{1} also coincides \cite[Proposition A.2]{HM}. However, it seems that the proof of \cite[Proposition A.2]{HM} contains a small gap. Thus we fix this gap. In addition, we consider general $d$-adic martingales, while the authors in \cite{HM} only deal with dyadic martingales on $\mathbb{R}^n$.

The paper is organized as follows. Section \ref{pre2} is devoted to notation and background, such as noncommutative $L_p$-spaces and $d$-adic martingales. In Section \ref{pthm2}, we will show Theorem \ref{2}. Our proof is based on Mei's method in \cite{Mei}. In Section \ref{pthm1}, we aim to prove Theorem \ref{1}. At the end, we will show Theorem \ref{3} in Section \ref{pthm3}. Our main ingredient is to use the duality that $BMO_\mathcal{M}^d(\mathbb{R})$ embeds continuously into the dual $(H^d_{1,\max}(\mathbb{R},\mathcal{M}))^*$ (see Subsection \ref{sec2.4}). 

Throughout this paper, we will use the following notation: $A\lesssim B$ (resp. $A\lesssim_\varepsilon B$) means that $A\le CB$ (resp. $A\le C_\varepsilon B$) for some absolute positive constant $C$ (resp. a positive constant $C_\varepsilon$ depending only on $\varepsilon$). $A\approx B$ or $A\approx_\varepsilon B$ means that these inequalities as well as their inverses hold. The index $p$ will be always assumed to be in $[1, \8]$. Denote by $p'=\frac{p}{p-1}$ the conjugate index  of $p$.

\bigskip

\section{Preliminaries}\label{pre2}

In this section, we provide notation and background that will be used throughout the paper.

\subsection{\textbf{Noncommutative $L_p$-spaces.}} Let $\mathcal{M}$ be a von Neumann algebra equipped with a normal semifinite faithful trace $\tau$. Denote by $\M_+$ the positive part of $\M$. Let $\mathcal{S}_+(\M)$ be the set of all $x\in \M_+$ whose support projection has a finite trace, and $\mathcal{S}(\mathcal{M})$ be the linear span of $\mathcal{S}_+(\mathcal{M})$. Then $\mathcal{S}(\mathcal{M})$ is a $w^*$-dense $*$-subalgebra of $\mathcal{M}$. Let $x\in\mathcal{S}(\mathcal{M})$, then $|x|^p\in\mathcal{S}(\mathcal{M})$ for any $0<p<\infty$, where $|x|:=(x^*x)^{1/2}$. Define
\[\|x\|_p=(\tau(|x|^p))^{1/p}.\]
Thus $\|\cdot\|_p$ is a norm for $p\ge 1$, and a $p$-norm for $0<p<1$. The noncommutative $L_p$-space associated with $(\mathcal{M},\tau)$ is the completion of $(\mathcal{S}(\M),\|\cdot\|_p)$ for $0<p<\8$  denoted by $L_p(\mathcal{M},\tau)$. We also write $L_p(\mathcal{M},\tau)$ simply by $L_p(\mathcal{M})$ for short. When $p=\8$, we set $L_\infty(\mathcal{M}):=\mathcal{M}$ equipped with the operator norm. In particular, when $p=2$, $L_2(\M)$ is a Hilbert space. We will view $\mathcal{M}$ as a von Neumann algebra on $L_2(\mathcal{M})$ by left multiplication, namely  $\M \hookrightarrow B(L_2(\M))$ via the embedding $x\longmapsto L_x\in B(L_2(\M))$, where $x\in \M$ and $L_x(y):= x\cdot y\in L_2(\M)$ for any $y\in L_2(\M)$. Hence in this way, $\mathcal{M}$ is in its standard form. It is well-known that for $1\leq p<\8$ and $p'=\frac{p}{p-1}$
$$  \big(L_p(\M)\big)^*=L_{p'}(\M).  $$
We refer the reader to \cite{PX} for a detailed exposition of noncommutative $L_p$-spaces.

\

Now we present the tensor product of von Neumann algebras. Assume that each $\mathcal{M}_k$ $(k=1, 2)$ is equipped with a normal semifinite faithful trace $\tau_k$. Then the tensor product of $\M_1$ and $\M_2$ denoted by $\mathcal{M}_1 \otimes \mathcal{M}_2$ is the $w^*$-closure of $\text{span}\{x_1\otimes x_2 | x_1\in \M_1, x_2\in \M_2\}$ in $B(L_2(\M_1)\otimes L_2(\M_2))$. Here $L_2(\M_1)\otimes L_2(\M_2)$ is the Hilbert space tensor product of $L_2(\M_1)$ and $ L_2(\M_2)$.

It is well-known that there exists a unique normal semifinite faithful trace $\tau$ on the von Neumann algebra tensor product $\mathcal{M}_1 \otimes \mathcal{M}_2$ such that
$$
\tau\left(x_1 \otimes x_2\right)=\tau_1\left(x_1\right) \tau_2\left(x_2\right), \quad \forall x_1 \in \mathcal{S}(\M_1), \forall x_2 \in \mathcal{S}(\M_2) .
$$
$\tau$ is called the tensor product of $\tau_1$ and $\tau_2$ and denoted by $\tau_1 \otimes \tau_2$.

\

In this paper, $\M$ will denote a semifinite von Neumann algebra equipped with a normal semifinite faithful trace $\tau$.

\subsection{$d$-adic martingales}\label{sec2.3}
Let $d\ge 2$ be a fixed integer. We are particularly interested in $d$-adic martingales since it is closely related to dyadic martingales on Euclidean spaces. In this section, we give a general definition of $d$-adic martingales. Afterwards we will present an orthonormal basis of Haar wavelets for $d$-adic martingales, which will be used to represent martingale paraproducts.

Let $\Omega$ be a measure space endowed with a $\sigma$-finite measure $\mu$. Assume that in $\Omega$, there exists a family of measurable sets $I_{n,k}$ for $n, k\in \mathbb{Z}$ satisfying the following properties:
\begin{enumerate}
	\item  $I_{n,k}$  are pairwise disjoint for any $k$ if $n$ is fixed;
	\item $\cup_{k\in \mathbb{Z}}I_{n,k}=\Omega$ for every $n$;
	\item $I_{n,k}=\cup_{q=1}^d I_{n+1,kd+q-1}$ for any $n, k$, so each $I_{n,k}$ is a union of $d$ disjoint subsets $I_{n+1,kd+q-1}$;
	\item $\mu(I_{n,k})=d^{-n}$ for any $n, k$.
\end{enumerate}
Then $I_{n,k}$ are called $d$-adic intervals, and let $\mathcal{ D}$ be the family of all such $d$-adic intervals.
Denote by $\mathcal{D}_n$ the collection of $d$-adic intervals of length $d^{-n}$ in $\mathcal{D}$.
For each $n\in \mathbb{Z}$, denote by $\mathcal{F}_n$ the $\sigma$-algebra generated by the $d$-adic intervals $I_{n,k}$, $\forall k\in \mathbb{Z}$. Denote by $\mathcal{F}$ the $\sigma$-algebra generated by all $d$-adic intervals for all $I_{n,k}$, $\forall n,k\in\mathbb{Z}$. 

Then $(\mathcal{F}_n)_{n\in \mathbb{Z}}$ is a filtration associated with the measure space $(\Omega, \mathcal{F}, \mu)$. Denote by $L^{\rm{loc}}_1(\Omega)$ the family of all locally integrable functions $g$ on $\Omega$, that is, $g\in L_1(I_{n, k})$ for all $n, k\in \mathbb{Z}$. For a locally integrable function $g\in L^{\rm{loc}}_1(\Omega)$, the sequence $(g_n)_{n\in\mathbb{Z}}$ is called a $d$-adic martingale, where
$$g_n= \mathbb{E}(g|\mathcal{F}_n)=\sum_{k=-\8}^{\8}\frac{\mathbbm{1}_{I_{n,k}}}{\mu(I_{n,k})}\int_{I_{n,k}}g\ d\mu.  $$
The martingale differences are defined as $d_n g=g_n-g_{n-1}$ for any $n\in\mathbb{Z}$. We also denote $g_n$ by $\mathbb{E}_n(g)$ ($n\in\mathbb{Z}$) as usual.

\begin{definition}\label{haar}
	Let $\omega=e^{\frac{2\pi \mathrm{i}}{d}}$ (here $\mathrm{i} $ is the imaginary number). For any $I=I_{n,k}\in \mathcal{D}$, define
	\[h_I^i=d^{n/2}\sum_{j=0}^{d-1}  \omega^{i{(j+1)}}\mathbbm{1}_{I_{n+1,kd+j}}, \quad \forall \ 1\leq i\leq d-1,\]
	and $h_I^0:=d^{n/2}\mathbbm{1}_I$.
\end{definition}
Then $\{h_I^i\}_{I\in\mathcal{D},1\leq i\leq d-1}$ is an orthonormal basis on $L_2(\Omega)$ because $\forall g\in L_2(\Omega)$
$$  g=\sum_{k=-\8}^{\8} d_kg=\sum_{k=-\8}^{\8} \biggl(\sum_{|I|=d^{-k+1}}\sum_{i=1}^{d-1}h_I^i\langle h_I^i,g\rangle\biggr). $$
We call $\{h_I^i\}_{I\in\mathcal{D},1\leq i\leq d-1}$  the system of Haar wavelets. Note that for any $1\leq i, j\leq d-1$,
\begin{equation}\label{formula2.2}
	h_I^i\cdot h_I^j=\mu(I)^{-1/2} h_I^{\overline{i+j}},
\end{equation}
where $\overline{i+j}$ is the remainder in $[1, d]$ modulo $d$.

\

As in the commutative setting, we can define the semicommutative $d$-adic martingales in the same way. Similarly, denote by $L^{\rm{loc}}_1({\mathbb{R}}, L_1(\M))$ the family of all $f$ such that $\mathbbm{1}_{I_{n, k}}\cdot f \in L_1({\mathbb{R}}, L_1(\M))$ for any $n, k\in\mathbb{Z}$. Then $\forall f\in L_1^{\rm{loc}}(\mathbb{R}, L_1(\M))$, the sequence $(f_n)_{n\in\mathbb{Z}}$ is called a semicommutative $d$-adic martingale, where
\begin{equation}\label{fn}
	f_n= \mathbb{E}(f|\mathcal{F}_n)=\sum_{k=-\8}^{\8}\frac{\mathbbm{1}_{I_{n,k}}}{\mu(I_{n,k})}\int_{I_{n,k}}f\ d\mu.  
\end{equation}
For any $f\in L_1(\mathbb{R},L_1(\mathcal{M}))$ and $g\in L_\8(\mathbb{R})$, define
\[\langle g, f\rangle= \int_{\mathbb{R}}\overline{g} \cdot f \ dm. \]
One can easily deduce that $\la g, f\ra \in L_1(\M)$ from the triangle inequality. By a slight abuse of notation, we use the same notation $\langle$$\cdot$,$\cdot$$\rangle$ to denote the inner product in any given Hilbert space. Besides, by \eqref{fn}, the martingale differences are given by $\forall f\in L_1^{\rm{loc}}(\mathbb{R}, L_1(\M)) \ \text{and} \ n\in \mathbb{Z}$
$$ d_n f=\sum_{|I|=d^{-n+1}}\sum_{i=1}^{d-1}h_I^i\otimes\langle h_I^i,f\rangle. $$

We will utilize $h_I^i$ to give a direct representation of $\pi_b$, which is easier to handle. It is well-known that $L_2(\mathbb{R},L_2(\mathcal{M}))=L_2(\mathbb{R})\otimes L_2(\mathcal{M})$. In the sequel, for any $f\in L_2(\mathbb{R})$ and $x\in L_2(\mathcal{M})$, we use ``$x\cdot f$ ''  (or ``$f\cdot x$ '') to denote $f\otimes x\in L_2(\mathbb{R},L_2(\mathcal{M}))$ for the sake of simplicity.

Now we calculate $\pi_b$. Let $b\in  L_1^{\rm{loc}}(\mathbb{R}, L_1(\M))$. For $f\in L_2(\mathbb{R},L_2(\mathcal{M}))$, one has 
\begin{equation}\label{calcu}
	\begin{aligned}
		\pi_b(f){}&=\sum_{k=-\infty}^{\infty}d_kb\cdot f_{k-1}\\&=\sum_{k=-\infty}^{\infty}\biggl(\sum_{|I|=d^{-k+1}}\sum_{i=1}^{d-1}h_I^i\otimes\langle h_I^i,b\rangle\biggr)\biggl(\sum_{|I|=d^{-k+1}}\mathbbm{1}_I\otimes\Bigl\langle \frac{\mathbbm{1}_I}{|I|},f\Bigr\rangle\biggr)\\&=\sum_{I\in \mathcal{D}}\sum_{i=1}^{d-1}h^i_I\otimes\langle h_I^i,b\rangle \Bigl\langle \frac{\mathbbm{1}_I}{|I|},f\Bigr\rangle,
	\end{aligned}
\end{equation}
which can be rewritten as
\begin{equation}\label{pib}
	\begin{aligned}
		\pi_b(f)=\sum_{I\in \mathcal{D}}\sum_{i=1}^{d-1}h^i_I\langle h_I^i,b\rangle \Bigl\langle \frac{\mathbbm{1}_I}{|I|},f\Bigr\rangle.
	\end{aligned}
\end{equation}
The adjoint operator of $\pi_b$ is given by $\forall f\in L_2(\mathbb{R},L_2(\mathcal{M}))$
\begin{equation}\label{pistar}
	\begin{aligned}
		\pi_b^*(f)&=\sum_{k\in\mathbb{Z}} \mathbb{E}_{k-1}(d_k b^* d_kf) \\
		&=\sum_{I\in \mathcal{D}}\sum_{i=1}^{d-1}\frac{\mathbbm{1}_I}{|I|}\langle h_I^i, b\rangle^* \langle h_I^i,f\rangle\\
		&=\sum_{I\in \mathcal{D}}\sum_{i=1}^{d-1}\frac{\mathbbm{1}_I}{|I|}\langle b,h_I^i\rangle \langle h_I^i,f\rangle. 
	\end{aligned}
\end{equation}
\begin{rem}\label{bb}
	For $d=2$, since $d_k b \cdot d_kf$ is $\mathcal{D}_{k-1}$-measurable, one has $(\pi_{b^*})^*=\varLambda_b$.
\end{rem}

\subsection{Operator-valued $BMO$ spaces}\label{sec2.4}

Let $b$ be an $\mathcal{M}$-valued function that is Bochner integrable on any $d$-adic interval, and define the following operator-valued $BMO$ sapces:
\begin{equation*}
	\begin{aligned}
		BMO_c^d(\mathbb{R},\mathcal{M}){}&=\biggl\{b:\|b\|_{BMO_c^d(\mathbb{R},\mathcal{M})}=\sup_{m\in\mathbb{Z}}\biggl\|\mathbb{E}_m\sum_{k=m}^\infty |d_kb|^2\biggr\|^{1/2}_{\mathcal{M}}<\infty\biggr\};\\
		BMO_r^d(\mathbb{R},\mathcal{M})	&=\biggl\{b:\|b\|_{BMO_r^d(\mathbb{R},\mathcal{M})}=\|b^*\|_{BMO_c^d(\mathbb{R},\mathcal{M})}<\infty\biggr\};\\
		BMO_{cr}^d(\mathbb{R},\mathcal{M})&=BMO_c^d(\mathbb{R},\mathcal{M})\cap	BMO_r^d(\mathbb{R},\mathcal{M}).
	\end{aligned}
\end{equation*}

\begin{rem}
	When $\mathcal{M}=\mathbb{C}$, the above three $BMO$ spaces are exactly the usual martingale $BMO^d(\mathbb{R})$.
\end{rem}
\begin{rem}
	From the martingale John-Nirenberg inequality in \cite{Ga}, we deduce that for $1\le p<\infty$,
	\begin{equation*}
		\|f\|_{BMO^d(\mathbb{R})}\approx_p \sup_{n\in\mathbb{Z}}\big\|\mathbb{E}_n|f-f_{n-1}|^p\big\|_\infty^{1/p}.
	\end{equation*}
\end{rem}

The following lemma is the interpolation result between $BMO$ and noncommutative $L_p$-spaces (see \cite[Theorem 3.11]{Mu}).
\begin{lem}\label{Musat}
	Let $1<p<\infty$ and $0<\theta<1$. Then
	\begin{equation*}
		\bigl(BMO_{cr}^d(\mathbb{R},\mathcal{M}),L_p(L_\infty(\mathbb{R})\otimes \mathcal{M})\bigr)_{p/q}=L_q(L_\infty(\mathbb{R})\otimes \mathcal{M}),\quad \text{with}\,\,q=p/\theta.
	\end{equation*}
\end{lem}

We conclude this section by noncommutative martingale Hardy spaces.
\begin{definition}
	For $g\in L_1(L_\infty(\mathbb{R})\otimes \mathcal{M})$, define
	\begin{enumerate}
		\item the $d$-adic martingale square function
		\begin{equation*}
			S(g)=\biggl(\sum_{k\in\mathbb{Z}} |d_kg|^2\biggr)^{1/2};
		\end{equation*}
		\item the $d$-adic martingale conditional square function
		\begin{equation*}
			\begin{aligned}
				s(g)=\biggl(\sum_{k\in\mathbb{Z}}\mathbb{E}_{k-1}\bigl(|d_kg|^2\bigr)\biggr)^{1/2}.
			\end{aligned}
		\end{equation*}
	\end{enumerate}
\end{definition}

\begin{rem}
	Note that $d$-adic martingales are regular. Then $\forall g\in L_1(L_\infty(\mathbb{R})\otimes \mathcal{M})$, $\forall 1\le p<\infty$  
	\begin{equation*}
		\|S(g)\|_{L_p(L_\infty(\mathbb{R})\otimes \mathcal{M})}\approx \|s(g)\|_{L_p(L_\infty(\mathbb{R})\otimes \mathcal{M})}.
	\end{equation*}
\end{rem}
Now we define the $d$-adic martingale column Hardy space $h_{p,c}^d(\mathbb{R})$
\begin{equation}
	\begin{aligned}
		h_{p,c}^d(\mathbb{R})=\bigg\{g\in  L_1(L_\infty(\mathbb{R})\otimes \mathcal{M}):\|g\|_{h_{p,c}^d(\mathbb{R})}=\|s(g)\|_{L_p(L_\infty(\mathbb{R})\otimes \mathcal{M})}<\infty\bigg\}.
	\end{aligned}
\end{equation}
\begin{rem}
	It is known that $(h_{p,c}^d(\mathbb{R}))^*=h_{p',c}^d(\mathbb{R})$ for $ 1<p<\infty$ in \cite{PX97}.
\end{rem}

In \cite{PX97}, Pisier and Xu showed the following noncommutative Burkholder-Gundy inequality:
\begin{thm}\label{BG}
	Let $1<p<\infty$.
	
$\mathrm{(1)}$
		 If $2\le p<\infty$, we have
		\begin{equation*}
			\|f\|_{L_p(L_\infty(\mathbb{R})\otimes \mathcal{M})}\approx_p \max\{\|S(f)\|_{L_p(L_\infty(\mathbb{R})\otimes \mathcal{M})},\|S(f^*)\|_{L_p(L_\infty(\mathbb{R})\otimes \mathcal{M})}\};
		\end{equation*}
	
$\mathrm{(2)}$ If $1< p< 2$, we have
		\begin{equation*}
			\|f\|_{L_p(L_\infty(\mathbb{R})\otimes \mathcal{M})}\approx_p \inf_{f=f_1+f_2}\{\|S(f_1)\|_{L_p(L_\infty(\mathbb{R})\otimes \mathcal{M})}+\|S(f_2^*)\|_{L_p(L_\infty(\mathbb{R})\otimes \mathcal{M})}\}.
		\end{equation*}

\end{thm}

By virtue of the vector-valued maximal function, the $d$-adic martingale Hardy space $H^d_{1,\max}(\mathbb{R},\mathcal{M})$ is defined by
\begin{equation*}
	H^d_{1,\max}(\mathbb{R},\mathcal{M})=\bigg\{g\in  L_1(L_\infty(\mathbb{R})\otimes \mathcal{M}):\|g\|_{H^d_{1,\max}(\mathbb{R},\mathcal{M})}:=\bigg\|\sup_{m\in\mathbb{Z}}\|\mathbb{E}_mg\|_{L_1(\mathcal{M})}\bigg\|_{L_1(\mathbb{R})}<\infty\bigg\}.
\end{equation*}
Bourgain and Garcia-Cuerva proved independently that $BMO_\mathcal{M}^d(\mathbb{R})$ embeds continuously into the dual $(	H^d_{1,\max}(\mathbb{R},\mathcal{M}))^*$. We refer the reader to \cite{BJ} for more details.

\bigskip

\section{Proof of Theorem \ref{2}}\label{pthm2}

Our proof follows the pattern set up in \cite{Mei}. We use the interpolation method and Lemma \ref{Musat}. Besides, we consider the general $d$-adic martingales. 
We first give the following lemma which is proved by Mei in \cite[Lemma 3.4]{Mei}. 
\begin{lem}\label{infcr}
	Let $1<p<\infty$. Let $b\in BMO_{r}^d(\mathbb{R},\mathcal{M})$. Assume that $\pi_b$ is bounded on $L_p(L_\infty(\mathbb{R})\otimes \mathcal{M})$. Then
	\begin{equation*}
		\|\pi_b\|_{L_\infty(L_\infty(\mathbb{R})\otimes \mathcal{M})\to BMO_{cr}^d(\mathbb{R},\mathcal{M})}\lesssim_{p} \|\pi_b\|_{L_p(L_\infty(\mathbb{R})\otimes \mathcal{M})\to L_p(L_\infty(\mathbb{R})\otimes \mathcal{M})}+\|b\|_{BMO_{r}^d(\mathbb{R},\mathcal{M})}.
	\end{equation*}
\end{lem}

The following proposition is about the boundedness of symmetric paraproduct $\pi_b+(\pi_{b^*})^*$, which was proved by Mei for $d=2$ in \cite{Mei2}, and  Hong, Liu and Mei for $d=2^n$ $(n\in \mathbb{N})$ in \cite{HM}. We just give a sketch of this.
\begin{proposition}\label{pibsumpibstar}
	Let $1<p<\infty$ and $b\in BMO^d_\mathcal{M}(\mathbb{R})$. Then $\pi_b+(\pi_{b^*})^*$ is bounded on $L_p(L_\infty(\mathbb{R})\otimes \mathcal{M})$ and 
	\begin{equation*}
		\|\pi_b+(\pi_{b^*})^*\|_{L_p(L_\infty(\mathbb{R})\otimes \mathcal{M})\to L_p(L_\infty(\mathbb{R})\otimes \mathcal{M})}\lesssim_{d, p} \|b\|_{BMO^d_\mathcal{M}(\mathbb{R})}.
	\end{equation*}
\end{proposition}
\begin{proof}
	By \eqref{pistar}, we get for any $f\in L_p(L_\infty(\mathbb{R})\otimes \mathcal{M})$ and $g\in L_{p'}(L_\infty(\mathbb{R})\otimes \mathcal{M})$, 
	\begin{equation*}
		\begin{aligned}
			\langle (\pi_b+(\pi_{b^*})^*)(f),g\rangle{}&=\biggl\langle \sum_{k\in\mathbb{Z}}d_kb\cdot f_{k-1}+\sum_{k\in\mathbb{Z}}\mathbb{E}_{k-1}(d_kb\cdot d_kf),g \biggr\rangle\\
			&=\sum_{k\in\mathbb{Z}}\langle d_kb,d_kg\cdot f_{k-1}^*\rangle+\sum_{k\in\mathbb{Z}}\langle d_kb,g_{k-1}\cdot d_kf^*\rangle\\
			&=\biggl\langle b,\sum_{k\in\mathbb{Z}}d_kg\cdot f_{k-1}^*+\sum_{k\in\mathbb{Z}}g_{k-1}\cdot d_kf^*\biggr\rangle.
		\end{aligned}
	\end{equation*}
	By direct calculations, one has
	\begin{equation*}
		\begin{aligned}
			\bigg\|\sum_{k\in\mathbb{Z}}d_kg\cdot f_{k-1}^*+\sum_{k\in\mathbb{Z}}g_{k-1}\cdot d_kf^*\bigg\|_{H^d_{1,\max}(\mathbb{R},\mathcal{M})}{}&=\bigg\|\sup_{m\in\mathbb{Z}}\Big\|\sum_{k\le m}d_kg\cdot f_{k-1}^*+\sum_{k\le m}g_{k-1}\cdot d_kf^*\Big\|_{L_1(\mathcal{M})}\bigg\|_{L_1(\mathbb{R})}\\
			&=\bigg\|\sup_{m\in\mathbb{Z}}\Big\|g_mf^*_{m}-\sum_{k\le m}d_kg\cdot d_kf^*\Big\|_{L_1(\mathcal{M})}\bigg\|_{L_1(\mathbb{R})}.
		\end{aligned}
	\end{equation*}
	Using the same method as in the proof of \cite[Theorem 1.1]{Mei2}, we deduce 
	\begin{equation*}
		\begin{aligned}
			\bigg\|\sum_{k\in\mathbb{Z}}d_kg\cdot f_{k-1}^*+\sum_{k\in\mathbb{Z}}g_{k-1}\cdot d_kf^*\bigg\|_{H^d_{1,\max}(\mathbb{R},\mathcal{M})}{}&
			\lesssim_{d, p} \|f\|_{L_p(L_\infty(\mathbb{R})\otimes \mathcal{M})}\|g\|_{L_{p'}(L_\infty(\mathbb{R})\otimes \mathcal{M})}.
		\end{aligned}
	\end{equation*}
	This implies 
	\begin{equation*}
		\begin{aligned}
			\big|\langle (\pi_b+(\pi_{b^*})^*)(f),g\rangle\big|\lesssim_{d, p} \|b\|_{BMO^d_\mathcal{M}(\mathbb{R})}\|f\|_{L_p(L_\infty(\mathbb{R})\otimes \mathcal{M})}\|g\|_{L_{p'}(L_\infty(\mathbb{R})\otimes \mathcal{M})}.
		\end{aligned}
	\end{equation*}
	Therefore
	\begin{equation*}
		\begin{aligned}
			\|\pi_b+(\pi_{b^*})^*\|_{L_p(L_\infty(\mathbb{R})\otimes \mathcal{M})\to L_p(L_\infty(\mathbb{R})\otimes \mathcal{M})}\lesssim_{d, p} \|b\|_{BMO^d_\mathcal{M}(\mathbb{R})}.
		\end{aligned}
	\end{equation*}
\end{proof}

\begin{rem}
	When $d=2$ and $p=2$, Proposition \ref{pibsumpibstar} is first shown in \cite{BP08}.
\end{rem}

With the help of Lemma \ref{infcr} and Proposition \ref{pibsumpibstar}, we now prove Theorem \ref{2}.
\begin{proof}[Proof of Theorem \ref{2}]
	By assumption, $\exists 1<p_0<\infty$, such that $\pi_b$ is bounded on $L_{p_0}(L_\infty(\mathbb{R})\otimes \mathcal{M})$.
	Note that $\|b\|_{BMO_{cr}^d(\mathbb{R},\mathcal{M})}\le \|b\|_{BMO^d_{\mathcal{M}}(\mathbb{R})}$. From Lemma \ref{infcr}, one has
	\begin{equation*}
		\begin{aligned}
			\|\pi_b\|_{L_\infty(L_\infty(\mathbb{R})\otimes \mathcal{M})\to BMO_{cr}^d(\mathbb{R},\mathcal{M})}{}&\lesssim_{p_0} \|\pi_b\|_{L_{p_0}(L_\infty(\mathbb{R})\otimes \mathcal{M})\to L_{p_0}(L_\infty(\mathbb{R})\otimes \mathcal{M})}+\|b\|_{BMO^d_{\mathcal{M}}(\mathbb{R})}<\infty.
		\end{aligned}
	\end{equation*}
	Thus due to Lemma \ref{Musat}, for any $p_0<p<\infty$ one has
	\begin{equation*}
		\begin{aligned}
			\|\pi_b\|_{L_p(L_\infty(\mathbb{R})\otimes \mathcal{M})\to L_p(L_\infty(\mathbb{R})\otimes \mathcal{M})}<\infty.
		\end{aligned}
	\end{equation*}
	Then by Proposition \ref{pibsumpibstar}, for any $1<p<p_0'$
	\begin{equation*}
		\begin{aligned}
			{}&\quad\|\pi_{b^*}\|_{L_p(L_\infty(\mathbb{R})\otimes \mathcal{M})\to L_p(L_\infty(\mathbb{R})\otimes \mathcal{M})}\\
			&=\|(\pi_{b^*})^*\|_{L_{p'}(L_\infty(\mathbb{R})\otimes \mathcal{M})\to L_{p'}(L_\infty(\mathbb{R})\otimes \mathcal{M})}\\
			&\le \|\pi_b+(\pi_{b^*})^*\|_{L_{p'}(L_\infty(\mathbb{R})\otimes \mathcal{M})\to L_{p'}(L_\infty(\mathbb{R})\otimes \mathcal{M})}+\|\pi_{b}\|_{L_{p'}(L_\infty(\mathbb{R})\otimes \mathcal{M})\to L_{p'}(L_\infty(\mathbb{R})\otimes \mathcal{M})}\\
			&\lesssim_{d, p}\|b\|_{BMO^d_{\mathcal{M}}(\mathbb{R})}+\|\pi_{b}\|_{L_{p'}(L_\infty(\mathbb{R})\otimes \mathcal{M})\to L_{p'}(L_\infty(\mathbb{R})\otimes \mathcal{M})}<\infty.
		\end{aligned}
	\end{equation*}
	Fix $1<p_1<p_0'$, then
	\begin{equation*}
		\begin{aligned}
			\|\pi_{b^*}\|_{L_{p_1}(L_\infty(\mathbb{R})\otimes \mathcal{M})\to L_{p_1}(L_\infty(\mathbb{R})\otimes \mathcal{M})}<\infty.
		\end{aligned}
	\end{equation*}
	Repeating the above argument with $b$ replaced by $b^*$ and $p_0$  by $p_1$, we deduce that $\pi_{b^*}$ is bounded on $L_{p'}(L_\infty(\mathbb{R})\otimes \mathcal{M})$ for any $p'>p_1$. Then passing to adjoint, we see that for any $1<p<p'_1$
	\begin{equation*}
		\begin{aligned}
			\|\pi_b\|_{L_p(L_\infty(\mathbb{R})\otimes \mathcal{M})\to L_p(L_\infty(\mathbb{R})\otimes \mathcal{M})}<\infty.
		\end{aligned}
	\end{equation*}
	Note that $p_1'>p_0$. Combining the above obtained results, we obtain the desired announcement: for any $1<p<\infty$
	\begin{equation*}
		\begin{aligned}
			\|\pi_b\|_{L_p(L_\infty(\mathbb{R})\otimes \mathcal{M})\to L_p(L_\infty(\mathbb{R})\otimes \mathcal{M})}<\infty.
		\end{aligned}
	\end{equation*}
\end{proof}
\begin{rem}
	If $b\in BMO^d(\mathbb{R})$, then $\pi_{b}$ is bounded in $L_2(L_\infty(\mathbb{R})\otimes \mathcal{M})$. Hence by Theorem \ref{2}, $\pi_{b}$ is bounded in $L_p(L_\infty(\mathbb{R})\otimes \mathcal{M})$ for any $1<p<\infty$.
\end{rem}

\bigskip

\section{Proof of Theorem \ref{1}}\label{pthm1}

 Now we concentrate on the boundedness of $\Theta_b$ which has been defined in \eqref{theta}. Hong, Liu and Mei have given an approach of Theorem \ref{1} but their proof seems to contain a gap. Thus Theorem \ref{1} and its proof have fixed this gap. Owing to Proposition \ref{pibsumpibstar}, it remains to discuss the boundedness of $\varLambda_b-(\pi_{b^*})^*$. The following lemma is elementary and will be needed in the proof of Proposition \ref{lambdabpib}.
\begin{lemma}\label{aibi}
	Assume that $(a_{I,i})_{I\in \mathcal{D}, 1\le i\le d-1}$ and $(b_{J,i})_{J\in \mathcal{D}, 1\le i\le d-1}$ are two sequences of operators in $\mathcal{M}$ indexed by the dyadic intervals. Then we have
	\begin{equation*}
		\sum_{I\in\mathcal{D}}\bigg|\sum_{i=1}^{d-1} a_{I,i}b_{I,i}\bigg|^2\le \Bigl(\sup_{I\in \mathcal{D}}\sum_{i=1}^{d-1}\|a_{I,i}\|^2_{\mathcal{M}}\Bigr)\cdot\sum_{J\in\mathcal{D}}\sum_{j=1}^{d-1} b_{J,j}^*b_{J,j}.
	\end{equation*}
\end{lemma}
\begin{proof}
Denote by $e_{i,j}$ the $(d-1)\times (d-1)$ matrix which has a 1 in the $(i,j)$-th position as its only nonzero entry. Then it is equivalent to prove
	\begin{equation*}
	\sum_{I\in\mathcal{D}}\bigg|\sum_{i=1}^{d-1} a_{I,i}b_{I,i}\otimes e_{1,1}\bigg|^2\le \Big(\sup_{I\in \mathcal{D}}\sum_{i=1}^{d-1}\|a_{I,i}\|^2_{\mathcal{M}}\Big)\cdot\sum_{J\in\mathcal{D}}\sum_{j=1}^{d-1} b_{J,j}^*b_{J,j}\otimes e_{1,1}.
\end{equation*}
Let $A_I=\sum\limits_{j=1}^{d-1} a_{I,j}\otimes e_{1,j}$ and $B_I=\sum\limits_{j=1}^{d-1} b_{I,j}\otimes e_{j,1}$. Then
	\begin{equation*}
		\begin{aligned}
			\sum_{I\in\mathcal{D}}\bigg|\sum_{i=1}^{d-1} a_{I,i}b_{I,i}\otimes e_{1,1}\bigg|^2{}&=\sum_{I\in\mathcal{D}}|A_IB_I|^2
			=\sum_{I\in\mathcal{D}}B_I^*A_I^*A_IB_I.
		\end{aligned}
	\end{equation*}
Note for any $I\in\mathcal{D}$, 
$$B_I^*A_I^*A_IB_I \le B_I^*B_I\|A_I^*A_I\|_{L_\infty(\mathbb{M}_{d-1}\otimes \mathcal{M})}.$$
Hence, we deduce that
\begin{equation*}
	\begin{aligned}
		\sum_{I\in\mathcal{D}}\bigg|\sum_{i=1}^{d-1} a_{I,i}b_{I,i}\otimes e_{1,1}\bigg|^2{}&\le \Big(\sup_{I\in \mathcal{D}}\|A_I^*A_I\|_{L_\infty(\mathbb{M}_{d-1}\otimes \mathcal{M})}\Bigr)\cdot\sum_{J\in\mathcal{D}}B_J^*B_J\\
		&=\Big(\sup_{I\in \mathcal{D}}\|A_IA_I^*\|_{L_\infty(\mathbb{M}_{d-1}\otimes \mathcal{M})}\Bigr)\cdot\sum_{J\in\mathcal{D}}B_J^*B_J \\
		&\le \Big(\sup_{I\in \mathcal{D}}\sum_{i=1}^{d-1}\|a_{I,i}\|^2_{\mathcal{M}}\Bigr)\cdot\sum_{J\in\mathcal{D}}\sum_{j=1}^{d-1} b_{J,j}^*b_{J,j}\otimes e_{1,1},
	\end{aligned}
\end{equation*}
as desired.
\end{proof}

By the above lemma, we can deal with the term $\varLambda_b-(\pi_{b^*})^*$.  In particular, when $d=2$, one has $\varLambda_b-(\pi_{b^*})^*=0$ (see Remark \ref{bb}).
\begin{proposition}\label{lambdabpib}
	Let $1<p<\infty$ and $b\in BMO^d_\mathcal{M}(\mathbb{R})$.
	
		$\mathrm{(1)}$ If $2\le p< \infty$, then $\varLambda_b-(\pi_{b^*})^*$ is bounded from $L_p(L_\infty(\mathbb{R})\otimes \mathcal{M})$ to $h_{p,c}^d(\mathbb{R})$ and
		\begin{equation*}
			\|\varLambda_b-(\pi_{b^*})^*\|_{L_p(L_\infty(\mathbb{R})\otimes \mathcal{M})\to h_{p,c}^d(\mathbb{R})}\lesssim_{d,p} \|b\|_{BMO^d_\mathcal{M}(\mathbb{R})};
		\end{equation*}
	
		$\mathrm{(2)}$ If $1< p\le 2$, then $\varLambda_b-(\pi_{b^*})^*$ is bounded from $h_{p,c}^d(\mathbb{R})$ to $L_p(L_\infty(\mathbb{R})\otimes \mathcal{M})$ and 
		\begin{equation*}
			\|\varLambda_b-(\pi_{b^*})^*\|_{h_{p,c}^d(\mathbb{R})\to L_p(L_\infty(\mathbb{R})\otimes \mathcal{M})}\lesssim_{d,p} \|b\|_{BMO^d_\mathcal{M}(\mathbb{R})}.
		\end{equation*}
\end{proposition}
\begin{proof}
(1) We aim to show 
\begin{equation*}
	\|(\varLambda_b-(\pi_{b^*})^*)(f)\|_{h_{p,c}^d(\mathbb{R})}\lesssim_{d,p} \|b\|_{BMO^d_\mathcal{M}(\mathbb{R})}\|f\|_{L_p(L_\infty(\mathbb{R})\otimes \mathcal{M})}.
\end{equation*}
From \eqref{rab} and \eqref{pistar} one has
\begin{equation*}
	\begin{aligned}
		(\varLambda_b-(\pi_{b^*})^*)(f)=\sum_{k\in\mathbb{Z}}d_kb\cdot d_kf-\sum_{k\in\mathbb{Z}}\mathbb{E}_{k-1}(d_kb\cdot d_kf)=\sum_{k\in\mathbb{Z}}d_{k}(d_kb\cdot d_kf).
	\end{aligned}
\end{equation*}
A direct calculation leads to
\begin{equation*}
	\begin{aligned}
		d_k(d_kb\cdot d_kf)=\sum_{I\in \mathcal{D}_{k-1}}\sum_{l=1}^{d-1}\sum_{\overline{i+j}=l}\langle h_I^i,b\rangle\langle h_I^j,f\rangle \frac{h_I^l}{|I|^{1/2}}.
	\end{aligned}
\end{equation*}
This implies 
\begin{equation*}
	\begin{aligned}
		|d_k(d_kb\cdot d_kf)|^2=\sum_{I\in \mathcal{D}_{k-1}}\Big|\sum_{l=1}^{d-1}\sum_{\overline{i+j}=l}\langle h_I^i,b\rangle\langle h_I^j,f\rangle \frac{h_I^l}{|I|^{1/2}}\Big|^2.
	\end{aligned}
\end{equation*}
Then we get
\begin{equation*}
	\begin{aligned}
		{}&\quad\mathbb{E}_{k-1}\big(|d_k(d_kb\cdot d_kf)|^2\big)\\
		&=\sum_{I\in \mathcal{ D}_{k-1}}\biggl\langle \frac{\mathbbm{1}_I}{|I|},\Big|\sum_{l=1}^{d-1}\sum_{\overline{i+j}=l}\langle h_I^i,b\rangle\langle h_I^j,f\rangle \frac{h_I^l}{|I|^{1/2}}\Big|^2\bigg\rangle \cdot\mathbbm{1}_I\\
		&=\sum_{I\in \mathcal{ D}_{k-1}}\biggl\langle \sum_{l=1}^{d-1}\sum_{\overline{i+j}=l}\langle h_I^i,b\rangle\langle h_I^j,f\rangle \frac{h_I^l}{|I|^{1/2}},\sum_{l=1}^{d-1}\sum_{\overline{i+j}=l}\langle h_I^i,b\rangle\langle h_I^j,f\rangle \frac{h_I^l}{|I|^{1/2}}\bigg\rangle_{L_2(\mathbb{R})} \cdot\frac{\mathbbm{1}_I}{|I|}\\
		&=\sum_{I\in\mathcal{D}_{k-1}}\sum_{l=1}^{d-1}\bigg|\sum_{\overline{i+j}=l}\langle h_I^i,b\rangle\langle h_I^j,f\rangle\cdot \frac{\mathbbm{1}_I}{|I|}\bigg|^2.
	\end{aligned}
\end{equation*}
Hence we have
\begin{equation*}
	\begin{aligned}
		\big[s\big((\varLambda_b-(\pi_{b^*})^*)(f)\big)\big]^2{}&=\sum_{k\in\mathbb{Z}}\mathbb{E}_{k-1}\big(|d_k(d_kb\cdot d_kf)|^2\big)
		=\sum_{I\in\mathcal{D}}\sum_{l=1}^{d-1}\Big|\sum_{i=1}^{d-1}\langle h_I^i,b\rangle\langle h_I^{\overline{l-i}},f\rangle\cdot \frac{\mathbbm{1}_I}{|I|}\Big|^2.
	\end{aligned}
\end{equation*}
Applying Lemma \ref{aibi} with $a_{I,i}=\frac{\langle h_I^i,b\rangle}{|I|^{1/2}}$ and $b_{I,i}=\frac{\langle h_I^{\overline{l-i}},f\rangle}{|I|^{1/2}}$, $\forall 1\le i\le d-1$, we get
\begin{equation*}
	\begin{aligned}
		s\big((\varLambda_b-(\pi_{b^*})^*)(f)\big)^2{}&
		\le (d-1) \sup_{I\in \mathcal{D}}\frac{1}{|I|}\sum_{i=1}^{d-1}\|\langle h_I^i,b\rangle\|^2_\mathcal{M}\cdot\sum_{J\in\mathcal{D}}\sum_{j=1}^{d-1}|\langle h_J^j,f\rangle|^2\frac{\mathbbm{1}_J}{|J|}\\
		&=(d-1) \sup_{I\in \mathcal{D}}\frac{1}{|I|}\sum_{i=1}^{d-1}\|\langle h_I^i,b\rangle\|^2_\mathcal{M}\cdot s(f)^2.
	\end{aligned}
\end{equation*}
Thus one has
\begin{equation*}
	\begin{aligned}
		\|s\big((\varLambda_b-(\pi_{b^*})^*)(f)\big)\|_{L_p(L_\infty(\mathbb{R})\otimes \mathcal{M})}{}&
		\lesssim_{d} \biggl(\sup_{I\in \mathcal{D}}\frac{1}{|I|^{1/2}}\sum_{i=1}^{d-1}\|\langle h_I^i,b\rangle\|_\mathcal{M}\biggr)\cdot \|s(f)\|_{L_p(L_\infty(\mathbb{R})\otimes \mathcal{M})}\\
		&\lesssim_p \biggl(\sup_{I\in \mathcal{D}}\frac{1}{|I|^{1/2}}\sum_{i=1}^{d-1}\|\langle h_I^i,b\rangle\|_\mathcal{M}\biggr)\cdot \|f\|_{L_p(L_\infty(\mathbb{R})\otimes \mathcal{M})}.
	\end{aligned}
\end{equation*}
However, 
\begin{equation*}
	\begin{aligned}
		\|\langle h_I^i,b\rangle\|_\mathcal{M}{}&=\Big\|\Bigl\langle h_I^i, b-\big\langle \frac{\mathbbm{1}_I}{|I|},b \big\rangle \Bigl\rangle\Big\|_{\mathcal{M}}\\
		&\le \frac{1}{|I|^{1/2}}\int_I \Big\| b(x)-\big\langle \frac{\mathbbm{1}_I}{|I|},b \big\rangle \Big\|_{\mathcal{M}}dx\\
		&\le \biggl(\int_I \Big\| b(x)-\big\langle \frac{\mathbbm{1}_I}{|I|},b \big\rangle \Big\|^2_{\mathcal{M}}dx\biggr)^{1/2}.
	\end{aligned}
\end{equation*}
This implies
\begin{equation*}
	\begin{aligned}
		\|(\varLambda_b-(\pi_{b^*})^*)(f)\|_{h_{p,c}^d(\mathbb{R})}\lesssim_{d,p} \|b\|_{BMO^d_\mathcal{M}(\mathbb{R})}\|f\|_{L_p(L_\infty(\mathbb{R})\otimes \mathcal{M})}.
	\end{aligned}
\end{equation*}
Therefore,
\begin{equation*}
	\|\varLambda_b-(\pi_{b^*})^*\|_{L_p(L_\infty(\mathbb{R})\otimes \mathcal{M})\to h_{p,c}^d(\mathbb{R})}\lesssim_{d,p} \|b\|_{BMO^d_\mathcal{M}(\mathbb{R})}.
\end{equation*}
(2)  Let $f\in h_{p,c}^d(\mathbb{R})$ and $g\in L_{p'}(\mathbb{R},L_{p'}(\mathcal{M}))$. Note that 
\begin{equation*}
	(\varLambda_b-(\pi_{b^*})^*)^*(g)=\sum_{k\in\mathbb{Z}}d_k(d_kb^*\cdot d_kg).
\end{equation*}
Hence, repeating the proof of $(1)$, one has
\begin{equation*}
	\|(\varLambda_b-(\pi_{b^*})^*)^*(g)\|_{h_{p'}^c(\mathbb{R})}\lesssim_{d,p} \|b\|_{BMO^d_\mathcal{M}(\mathbb{R})}\|g\|_{L_{p'}(L_\infty(\mathbb{R})\otimes \mathcal{M})}.
\end{equation*}
This implies
\begin{equation*}
	\begin{aligned}
		|\langle (\varLambda_b-(\pi_{b^*})^*)(f),g\rangle|{}&=|\langle f,(\varLambda_b-(\pi_{b^*})^*)^*g\rangle|\\
		&\lesssim_{d,p} \|f\|_{h_{p,c}^d(\mathbb{R})}\|b\|_{BMO^d_\mathcal{M}(\mathbb{R})}\|g\|_{L_{p'}(L_\infty(\mathbb{R})\otimes \mathcal{M})}.
	\end{aligned}
\end{equation*}
Therefore,
\begin{equation*}
	\|\varLambda_b-(\pi_{b^*})^*\|_{h_{p,c}^d(\mathbb{R})\to L_p(L_\infty(\mathbb{R})\otimes \mathcal{M})}\lesssim_{d,p} \|b\|_{BMO^d_\mathcal{M}(\mathbb{R})}.
\end{equation*}
\end{proof}

 Now we prove Theorem \ref{1}.

\begin{proof}[Proof of Theorem \ref{1}]
(1)	If $2\le p<\infty$, by the triangle inequality, one has
	\begin{equation*}
		\begin{aligned}
			\|\Theta_b\|_{L_p(L_\infty(\mathbb{R})\otimes \mathcal{M})\to h_{p,c}^d(\mathbb{R})}
			{}&\le \|\pi_b+(\pi_{b^*})^*\|_{L_p(L_\infty(\mathbb{R})\otimes \mathcal{M})\to h_{p,c}^d(\mathbb{R})}\\
			&\quad+\|\varLambda_b-(\pi_{b^*})^*\|_{L_p(L_\infty(\mathbb{R})\otimes \mathcal{M})\to h_{p,c}^d(\mathbb{R})}.
		\end{aligned}
	\end{equation*}
From Proposition \ref{pibsumpibstar}, for any $f\in L_p(L_\infty(\mathbb{R})\otimes \mathcal{M})$, we have
\begin{equation*}
	\begin{aligned}
		\|(\pi_b+(\pi_{b^*})^*)(f)\|_{h_{p,c}^d(\mathbb{R})}{}&\lesssim_p \|(\pi_b+(\pi_{b^*})^*)(f)\|_{L_p(L_\infty(\mathbb{R})\otimes \mathcal{M})} \lesssim_{d,p}\|b\|_{BMO^d_\mathcal{M}(\mathbb{R})}\|f\|_{L_p(L_\infty(\mathbb{R})\otimes \mathcal{M})}.
	\end{aligned}
\end{equation*}
This implies that
\begin{equation*}
	\begin{aligned}
		\|\pi_b+(\pi_{b^*})^*\|_{L_p(L_\infty(\mathbb{R})\otimes \mathcal{M})\to h_{p,c}^d(\mathbb{R})}\lesssim_{d,p}\|b\|_{BMO^d_\mathcal{M}(\mathbb{R})}\|f\|_{L_p(L_\infty(\mathbb{R})\otimes \mathcal{M})}.
	\end{aligned}
\end{equation*}
Therefore from Proposition \ref{lambdabpib}
\begin{equation*}
	\|\Theta_b\|_{L_p(L_\infty(\mathbb{R})\otimes \mathcal{M})\to h_{p,c}^d(\mathbb{R})}\lesssim_{d,p} \|b\|_{BMO^d_\mathcal{M}(\mathbb{R})},
\end{equation*}
as desired.
(2) can be shown similarly.
\end{proof}

\bigskip

\section{Proof of Theorem \ref{3}}\label{pthm3}
Let $X$ be a Banach space. For $1\le p<\infty$, let $L_p(\mathbb{R},X)$ be the Banach space of all strongly Lebesgue measurable functions $f$ from $\mathbb{R}$ to $X$ for which
\begin{equation*}
	\|f\|_{L_p(\mathbb{R},X)}=\biggl(\int_\mathbb{R} \|f\|_X^p dm\biggr)^{1/p}<\infty.
\end{equation*}

The following lemma will be used in the proof of Theorem \ref{3}.
\begin{lem}\label{af}
	Let $1<p<\infty$ and $X$ be a Banach space. If $a\in BMO^d(\mathbb{R})$ and $f\in L_p(\mathbb{R},X)$, then
	\begin{equation*}
		\begin{aligned}
			\biggl\|\sup_{m\in\mathbb{Z}}\Bigl\|\mathbb{E}_{m}\Bigl(\sum_{j\ge m+1}d_ja\cdot d_jf\Bigr)\Bigr\|_{X}\biggl\|_{L_p(\mathbb{R})}\lesssim_{p}\|a\|_{BMO^d(\mathbb{R})}\|f\|_{L_p(\mathbb{R},X)}.
		\end{aligned}
	\end{equation*}
\end{lem}
\begin{proof}
	Let $1<r=\frac{p+1}{2}<p$. Using the H\"{o}lder inequality and the martingale John-Nirenberg inequality, we obtain
	\begin{equation*}
		\begin{aligned}
			\biggl\|\mathbb{E}_{m}\Bigl(\sum_{j\ge m+1}d_ja\cdot d_jf\Bigr)\biggr\|_{X}{}
			&=\big\|\mathbb{E}_{m}\big((a-a_m)(f-f_m)\big)\big\|_{X}\\
			&\le\mathbb{E}_{m}\big(|a-a_m|\cdot\|f-f_m\|_{X}\big)\\
			&\le \bigl(\mathbb{E}_{m}\big(|a-a_m|^{r'}\big)\bigr)^{1/{r'}}\bigl(\mathbb{E}_{m}\big(\|f-f_m\|_{X}^{r}\big)\bigr)^{1/r}\\
			&\lesssim_p \|a\|_{BMO^d(\mathbb{R})}\biggl(\bigl(\mathbb{E}_{m}\|f\|_{X}^{r}\bigr)^{1/r}+\bigl(\mathbb{E}_{m}\|f_m\|_{X}^{r}\bigr)^{1/r}\biggr)\\
			&=\|a\|_{BMO^d(\mathbb{R})}\cdot\biggl(\bigl(\mathbb{E}_{m}\|f\|_{X}^{r}\bigr)^{1/r}+\|f_m\|_{X}\biggr).
		\end{aligned}
	\end{equation*}
 Note that $p/r>1$. Hence, this implies that
\begin{equation*}
	\begin{aligned}
	     {}&\quad\biggl\|\sup_{m\in\mathbb{Z}}\Bigl\|\mathbb{E}_{m}\Bigl(\sum_{j\ge m+1}d_ja\cdot d_jf\Bigr)\Bigr\|_{X}\biggl\|_{L_p(\mathbb{R})}\\
	     &\lesssim_p \|a\|_{BMO^d(\mathbb{R})}\cdot\biggl\|\sup_{m\in\mathbb{Z}}\Bigl(\bigl(\mathbb{E}_{m}\|f\|_{X}^{r}\bigr)^{1/r}+\|f_m\|_{X}\Bigr)\bigg\|_{L_{p}(\mathbb{R})}\\
		&\le \|a\|_{BMO^d(\mathbb{R})}\cdot \biggl(\Bigl\|\sup_{m\in\mathbb{Z}}\mathbb{E}_{m}\|f\|_{X}^{r}\Big\|^{1/r}_{L_{p/r}(\mathbb{R})}+\Bigl\|\sup_{m\in\mathbb{Z}}\|f_m\|_{X}\Big\|_{L_{p}(\mathbb{R})}\biggr)\\
		&\lesssim_p \|a\|_{BMO^d(\mathbb{R})}\cdot \bigg(\Big\|\|f\|_{X}^{r}\Big\|^{1/r}_{L_{p/r}(\mathbb{R})}+\|f\|_{L_p(\mathbb{R},X)}\bigg)\\
		&=  2\|a\|_{BMO^d(\mathbb{R})}\|f\|_{L_p(\mathbb{R},X)},
	\end{aligned}
\end{equation*}
where in the second inequality we use the triangle inequality, and the third inequality is from the vector-valued Doob maximal inequality.
\end{proof}

Finally, we prove Theorem \ref{3}.
\begin{proof}[Proof of Theorem \ref{3}]
For any $ f\in L_p(L_\infty(\mathbb{R})\otimes \mathcal{M})$,
\begin{equation}\label{piarb}
	\begin{aligned}
		[\pi_{a}, R_b](f){}&=\pi_{a} (R_b(f))-R_b (\pi_{a}(f))\\
		&=\sum_{k\in\mathbb{Z}} d_ka\cdot\mathbb{E}_{k-1} \biggl(\sum_{j\in\mathbb{Z}} b_{j-1}\cdot d_jf \biggr)- \sum_{k\in\mathbb{Z}} b_{k-1} \cdot d_k \biggl(\sum_{j\in\mathbb{Z}} d_ja\cdot f_{j-1}\biggr)\\
		&=\sum_{k\in\mathbb{Z}} d_ka\cdot \biggl(\sum_{j\le k-1}b_{j-1}\cdot d_jf\biggr)-\sum_{k\in\mathbb{Z}} b_{k-1}\cdot d_ka\cdot f_{k-1}\\
		&=\sum_{k\in\mathbb{Z}} d_ka\cdot \biggl(\sum_{j\le k-1}b_{j-1}\cdot d_jf- b_{k-1}\cdot f_{k-1}\biggr)\\
		&=-\sum_{k\in\mathbb{Z}} d_ka\cdot \biggl( \sum_{j\le k-1}d_jb\cdot d_jf\biggr)-\sum_{k\in\mathbb{Z}} d_ka\cdot \biggl(\sum_{j\le k-1}d_jb\cdot f_{j-1}\biggr)\\
		&=-\sum_{k\in\mathbb{Z}} d_ka\cdot \biggl( \sum_{j\le k-1}d_jb\cdot d_jf\biggr)-\pi_{a}(\pi_b(f)).
	\end{aligned}
\end{equation}
Define $ \forall f\in L_p(L_\infty(\mathbb{R})\otimes \mathcal{M})$
\begin{equation*}
	V_{a,b}(f)=\sum_{k\in\mathbb{Z}} d_ka\cdot \mathbb{E}_{k-1}\biggl(\sum_{j\ge k} d_j{b}\cdot d_j{f}\biggr).
\end{equation*}
Then
	\begin{equation*}
	\begin{aligned}
		[\pi_{a}, R_b](f){}&=-\pi_a(\varLambda_{{b}}({f}))+V_{a,b}(f)-\pi_{a}(\pi_b(f))=-\pi_{a}(\Theta_b(f))+V_{a,b}(f).
	\end{aligned}
\end{equation*}
This implies that
\begin{equation}\label{piamb}
	\begin{aligned}
		[\pi_a,M_b]=[\pi_a,\Theta_b]+[\pi_{a}, R_b]{}&=\pi_a\Theta_b-\Theta_b\pi_a-\pi_{a}\Theta_b+V_{a,b}\\
		&=-\Theta_b\pi_a+V_{a,b}\\
		&=-(\pi_b+(\pi_{b^*})^*)\pi_a-(\varLambda_b-(\pi_{b^*})^*)\pi_a+V_{a,b}.
	\end{aligned}
\end{equation}
From Proposition \ref{pibsumpibstar} and Theorem \ref{2} it follows that
\begin{equation}
	\begin{aligned}
		{}&\|(\pi_b+(\pi_{b^*})^*)\pi_a\|_{L_p(L_\infty(\mathbb{R})\otimes \mathcal{M})\to L_p(L_\infty(\mathbb{R})\otimes \mathcal{M})}\\
		&\le \|\pi_b+(\pi_{b^*})^*\|_{L_p(L_\infty(\mathbb{R})\otimes \mathcal{M})\to L_p(L_\infty(\mathbb{R})\otimes \mathcal{M})}\|\pi_a\|_{L_p(L_\infty(\mathbb{R})\otimes \mathcal{M})\to L_p(L_\infty(\mathbb{R})\otimes \mathcal{M})}\\
		&\lesssim_{d,p} \|a\|_{BMO^d(\mathbb{R})}\|b\|_{BMO^d_{\mathcal{M}}(\mathbb{R})}.
	\end{aligned}
\end{equation}
	Note that for any $f\in L_p(L_\infty(\mathbb{R})\otimes \mathcal{M})$ and $g\in L_{p'}(L_\infty(\mathbb{R})\otimes \mathcal{M})$, 
	\begin{equation*}
		\begin{aligned}
		    \langle V_{a,b}(f),g\rangle{}
			&=\sum_{k\in\mathbb{Z}}\biggl\langle d_ka\cdot \mathbb{E}_{k-1}\Bigl(\sum_{j\ge k}d_jb\cdot d_jf\Bigr),g\biggr\rangle\\
			&=\sum_{k\in\mathbb{Z}}\biggl\langle \sum_{j\ge k}d_jb\cdot d_jf, \mathbb{E}_{k-1}(d_k{\overline{a}}\cdot d_kg)\biggr\rangle\\
			&=\sum_{k\in\mathbb{Z}}\biggl\langle d_kb,\sum_{j\le k}\mathbb{E}_{j-1}(d_j\overline{a}\cdot d_jg)\cdot d_kf^*  \biggr\rangle\\
			&=\biggl\langle b,\sum_{k\in\mathbb{Z}}\sum_{j\le k}\mathbb{E}_{j-1}(d_j\overline{a}\cdot d_jg)\cdot d_kf^*  \biggr\rangle.
		\end{aligned}
	\end{equation*}
On the other hand,
\begin{equation*}
	\begin{aligned}
		\big\langle (\varLambda_b-(\pi_{b^*})^*)\pi_a(f),g\big\rangle{}
		&=\biggl\langle \sum_{k\in\mathbb{Z}}d_k(d_kb\cdot d_k(\pi_a(f))),g\biggr\rangle\\
		&=\sum_{k\in\mathbb{Z}}\langle d_kb\cdot d_ka\cdot f_{k-1},d_kg\rangle\\
		&=\biggl\langle b, \sum_{k\in\mathbb{Z}}d_k(d_k\overline{a}\cdot d_kg )\cdot f_{k-1}^*\bigg\rangle.
	\end{aligned}
\end{equation*}
Hence
	\begin{equation}\label{vabf}
		\begin{aligned}
			\big\langle V_{a,b}(f)-(\varLambda_b-(\pi_{b^*})^*)\pi_a(f),g\big\rangle
			{}&=\langle b, W_{a,f,g}\rangle,
		\end{aligned}
	\end{equation}
where 
\begin{equation*}
	\begin{aligned}
		W_{a,f,g}:=\sum_{k\in\mathbb{Z}}\sum_{j\le k}\mathbb{E}_{j-1}(d_j\overline{a}\cdot d_jg)\cdot d_kf^*-\sum_{k\in\mathbb{Z}}d_k(d_k\overline{a}\cdot d_kg)\cdot f_{k-1}^*.
	\end{aligned}
\end{equation*}
We have
	\begin{equation*}
		\begin{aligned}
			|\langle b, W_{a,f,g}\rangle|\lesssim \|b\|_{BMO^d_{\mathcal{M}}(\mathbb{R})}\|W_{a,f,g}\|_{H^d_{1,\max}(\mathbb{R},\mathcal{M})}.
		\end{aligned}
	\end{equation*}
We calculate directly that for any $m\in \mathbb{Z}$,
\begin{equation*}
	\begin{aligned}
		&{}\quad\mathbb{E}_m(W_{a,f,g})\\
		&=\sum_{k\le m}\sum_{j\le k}\mathbb{E}_{j-1}(d_j\overline{a}\cdot d_jg)\cdot d_kf^*-\sum_{k\le m}d_k(d_k\overline{a}\cdot d_kg)\cdot f_{k-1}^*\\
		&=\sum_{j\le m}\mathbb{E}_{j-1}(d_j\overline{a}\cdot d_jg)\cdot (f_m^*-f_{j-1}^*)-\sum_{j\le m}d_j(d_j\overline{a}\cdot d_jg)\cdot f_{j-1}^*\\
		&=\sum_{j\le m}\mathbb{E}_{j-1}(d_j\overline{a}\cdot d_jg)\cdot f_m^*-\sum_{j\le m}(d_j\overline{a}\cdot d_jg)\cdot f_{j-1}^*\\
		&=\mathbb{E}_{m}\biggl(\sum_{j\le m}\mathbb{E}_{j-1}(d_j\overline{a}\cdot d_jg)\biggr)\cdot f_m^*-\sum_{j\le m}(d_j\overline{a}\cdot d_jg)\cdot f_{j-1}^*\\
		&=\mathbb{E}_{m}\biggl(\sum_{j\in\mathbb{Z}}\mathbb{E}_{j-1}(d_j\overline{a}\cdot d_jg)\biggr)\cdot f_m^*-\mathbb{E}_{m}\biggl(\sum_{j\ge m+1}\mathbb{E}_{j-1}(d_j\overline{a}\cdot d_jg)\biggr)\cdot f_m^*-\sum_{j\le m}(d_j\overline{a}\cdot d_jg)\cdot f_{j-1}^*\\
		&=\mathbb{E}_{m}\biggl(\sum_{j\in\mathbb{Z}}\mathbb{E}_{j-1}(d_j\overline{a}\cdot d_jg)\biggr)\cdot f_m^*-\mathbb{E}_{m}\biggl(\sum_{j\ge m+1}d_j\overline{a}\cdot d_jg\biggr)\cdot f_m^*-\sum_{j\le m}(d_j\overline{a}\cdot d_jg)\cdot f_{j-1}^*.
	\end{aligned}
\end{equation*}
Hence
	\begin{equation*}
		\begin{aligned}
			{}&\quad\|W_{a,f,g}\|_{H^d_{1,\max}(\mathbb{R},\mathcal{M})}\\
			&=\bigg\|\sup_{m\in\mathbb{Z}}\|\mathbb{E}_m(W_{a,f,g})\|_{L_1(\mathcal{M})}\bigg\|_{L_1(\mathbb{R})}\\
			&\le \bigg\|\sup_{m\in\mathbb{Z}}\Bigl\|\mathbb{E}_{m}\Bigl(\sum_{j\in\mathbb{Z}}\mathbb{E}_{j-1}(d_j\overline{a}\cdot d_jg)\Bigr)\cdot f_m^*\Bigr\|_{L_1(\mathcal{M})}\bigg\|_{L_1(\mathbb{R})}\\
			&\quad+\bigg\|\sup_{m\in\mathbb{Z}}\Bigl\|\mathbb{E}_{m}\Bigl(\sum_{j\ge m+1}d_j\overline{a}\cdot d_jg\Bigr)\cdot f_m^*\Bigr\|_{L_1(\mathcal{M})}\bigg\|_{L_1(\mathbb{R})}+\bigg\|\sup_{m\in\mathbb{Z}}\Bigl\|\sum_{j\le m}(d_j\overline{a}\cdot d_jg)\cdot f_{j-1}^*\Bigr\|_{L_1(\mathcal{M})}\bigg\|_{L_1(\mathbb{R})}\\
			&:=\text{(I)}+\text{(II)}+\text{(III)}.
		\end{aligned}
	\end{equation*}

For the term $(\text{I})$, from \eqref{pistar}, we have 
$$  \sum_{j\in\mathbb{Z}}\mathbb{E}_{j-1}(d_j\overline{a}\cdot d_jg)=(\pi_a)^*(g).  $$
Thus 
\begin{equation}\label{I}
	\begin{aligned}
		(\text{I}){}&\le \biggl\|\sup_{m\in\mathbb{Z}}\Bigl\|\mathbb{E}_{m}((\pi_a)^*(g))\Bigr\|_{L_{p'}(\mathcal{M})}\cdot \sup_{m\in\mathbb{Z}}\|f_m\|_{L_p(\mathcal{M})}\biggr\|_{L_1(\mathbb{R})}\\
		&\le \biggl\|\sup_{m\in\mathbb{Z}}\Bigl\|\mathbb{E}_{m}((\pi_a)^*(g))\Bigr\|_{L_{p'}(\mathcal{M})}\biggr\|_{L_{p'}(\mathbb{R})}\cdot \biggl\|\sup_{m\in\mathbb{Z}}\|f_m\|_{L_p(\mathcal{M})}\biggr\|_{L_p(\mathbb{R})}\\
		&\lesssim_p \|(\pi_{a})^*(g)\|_{L_{p'}(L_\infty(\mathbb{R})\otimes \mathcal{M})}\|f\|_{L_p(L_\infty(\mathbb{R})\otimes \mathcal{M})}\\
		&\lesssim_{d,p} \|a\|_{BMO^d(\mathbb{R})}\|g\|_{L_{p'}(L_\infty(\mathbb{R})\otimes \mathcal{M})}\|f\|_{L_p(L_\infty(\mathbb{R})\otimes \mathcal{M})},
	\end{aligned}
\end{equation}
where the first inequality and the second inequality are both due to the H\"{o}lder inequality, and the third is from the  vector-valued Doob maximal inequality for $L_p(\mathcal{M})$-valued functions.

For the term $\text{(II)}$, we calculate that
\begin{equation}\label{II}
	\begin{aligned}
		\text{(II)}{}&\le \bigg\|\sup_{m\in\mathbb{Z}}\|f_m\|_{L_p(\mathcal{M})}\bigg\|_{L_{p}(\mathbb{R})}\cdot \bigg\|\sup_{m\in\mathbb{Z}}\Big\|\mathbb{E}_{m}\Bigl(\sum_{j\ge m+1}d_j\overline{a}\cdot d_jg\Bigr)\Big\|_{L_{p'}(\mathcal{M})}\bigg\|_{L_{p'}(\mathbb{R})}\\
		&\lesssim_p \|f\|_{L_p(L_\infty(\mathbb{R})\otimes \mathcal{M})}\|a\|_{BMO^d(\mathbb{R})}\|g\|_{L_{p'}(L_\infty(\mathbb{R})\otimes \mathcal{M})},
	\end{aligned}
\end{equation}
where for the first inequality we have used the H\"{o}lder inequality twice, the second is from the vector-valued Doob maximal inequality for $L_p(\mathcal{M})$-valued functions and Lemma \ref{af}.

For the term $\text{(III)}$, note that
$$ \sum_{j\le m}(d_j\overline{a}\cdot d_jg)\cdot f_{j-1}^*=\sum_{j\le m}d_jg\cdot (d_j\overline{a}\cdot f_{j-1}^*)=\sum_{j\le m}d_jg\cdot (d_j(\pi_a(f))^*).  $$
If $2\le p<\infty$, due to Theorem \ref{BG}, $\forall \varepsilon>0$, we choose $g_1$ and $g_2$ such that $g=g_1+g_2$ and
\begin{equation*}
	\begin{aligned}
		\|S(g_1)\|_{L_{p'}(L_\infty(\mathbb{R})\otimes \mathcal{M})}+\|S(g_2^*)\|_{L_{p'}(L_\infty(\mathbb{R})\otimes \mathcal{M})}\lesssim_p \|g\|_{L_{p'}(L_\infty(\mathbb{R})\otimes \mathcal{M})}+\varepsilon.
	\end{aligned}
\end{equation*}
This implies that by the H\"{o}lder inequality
	\begin{equation*}
		\begin{aligned}
			(\text{III}){}&\leq \bigg\|\sup_{m\in\mathbb{Z}}\Bigl\|\sum_{j\le m}d_jg_1\cdot (d_j(\pi_a(f))^*)\Bigr\|_{L_1(\mathcal{M})}\bigg\|_{L_1(\mathbb{R})}+\bigg\|\sup_{m\in\mathbb{Z}}\Bigl\|\sum_{j\le m}d_jg_2\cdot (d_j((\pi_a(f))^*))\Bigr\|_{L_1(\mathcal{M})}\bigg\|_{L_1(\mathbb{R})}\\
			&= \biggl\| \|S(g_1)\|_{L_{p'}(\mathcal{M})}\|S(\pi_a(f))\|_{L_p(\mathcal{M})}\biggr\|_{L_1(\mathbb{R})}+\biggl\| \|S(g_2^*)\|_{L_{p'}(\mathcal{M})}\|S((\pi_a(f))^*)\|_{L_p(\mathcal{M})}\biggr\|_{L_1(\mathbb{R})}\\
			&\le \|S(g_1)\|_{L_{p'}(L_\infty(\mathbb{R})\otimes \mathcal{M})}\|S(\pi_a(f))\|_{L_p(L_\infty(\mathbb{R})\otimes \mathcal{M})}+\|S(g_2^*)\|_{L_{p'}(L_\infty(\mathbb{R})\otimes \mathcal{M})}\|S((\pi_a(f))^*)\|_{L_p(L_\infty(\mathbb{R})\otimes \mathcal{M})}\\
			&\lesssim_{p} (\|g\|_{L_{p'}(L_\infty(\mathbb{R})\otimes \mathcal{M})}+\varepsilon)\cdot \|\pi_a(f)\|_{L_p(L_\infty(\mathbb{R})\otimes \mathcal{M})}\\
			&\lesssim_{d,p} (\|g\|_{L_{p'}(L_\infty(\mathbb{R})\otimes \mathcal{M})}+\varepsilon)\cdot \|a\|_{BMO^d(\mathbb{R})}\|f\|_{L_p(L_\infty(\mathbb{R})\otimes \mathcal{M})},
		\end{aligned}
	\end{equation*}
where in the third inequality we use the H\"{o}lder inequality. Letting $\varepsilon\to 0$, we get
\begin{equation}\label{III}
	\begin{aligned}
		(\text{III})\lesssim_{d,p} \|g\|_{L_{p'}(L_\infty(\mathbb{R})\otimes \mathcal{M})}\|a\|_{BMO^d(\mathbb{R})}\|f\|_{L_p(L_\infty(\mathbb{R})\otimes \mathcal{M})}.
	\end{aligned}
\end{equation}
For the case when $1<p\le 2$, we exchange $f$ with $g$ and repeat the above argument, then \eqref{III} can be shown similarly.
Hence from \eqref{I}, \eqref{II} and \eqref{III} we have
\begin{equation}\label{vab2}
	\begin{aligned}
		|\langle b, W_{a,f,g}\rangle|{}&\le \|b\|_{BMO^d_{\mathcal{M}}(\mathbb{R})}\|W_{a,f,g}\|_{H^d_{1,\max}(\mathbb{R},\mathcal{M})}\\
		&\lesssim_{d,p} \|a\|_{BMO^d(\mathbb{R})}\|b\|_{BMO^d_{\mathcal{M}}(\mathbb{R})}\|g\|_{L_{p'}(L_\infty(\mathbb{R})\otimes \mathcal{M})}\|f\|_{L_{p}(\mathbb{R},L_p(\mathcal{M}))}.
	\end{aligned}
\end{equation}
Due to \eqref{vabf} and \eqref{vab2} we deduce that
\begin{equation*}
	\begin{aligned}
		|\big\langle V_{a,b}(f)-(\varLambda_b-(\pi_{b^*})^*)\pi_a(f),g\big\rangle|\lesssim_{d,p} \|a\|_{BMO^d(\mathbb{R})}\|b\|_{BMO^d_{\mathcal{M}}(\mathbb{R})}\|g\|_{L_{p'}(L_\infty(\mathbb{R})\otimes \mathcal{M})}\|f\|_{L_p(L_\infty(\mathbb{R})\otimes \mathcal{M})},
	\end{aligned}
\end{equation*}
which yields
\begin{equation*}
	\begin{aligned}
		\|V_{a,b}\|_{L_p(L_\infty(\mathbb{R})\otimes \mathcal{M})\to L_p(L_\infty(\mathbb{R})\otimes \mathcal{M})}\lesssim_{d,p} \|a\|_{BMO^d(\mathbb{R})}\|b\|_{BMO^d_{\mathcal{M}}(\mathbb{R})}.
	\end{aligned}
\end{equation*}
Therefore, we conclude
\begin{equation*}
	\begin{aligned}
			\|[\pi_a,M_b]\|_{L_p(L_\infty(\mathbb{R})\otimes \mathcal{M})\to L_p(L_\infty(\mathbb{R})\otimes \mathcal{M})}\lesssim_{d,p} \|a\|_{BMO^d(\mathbb{R})}\|b\|_{BMO^d_{\mathcal{M}}(\mathbb{R})}.
	\end{aligned}
\end{equation*}

\end{proof}

\begin{rem}\label{Tstar2est}
	Let $1<p<\infty$. If $a\in BMO^d(\mathbb{R})$ and $b\in BMO^d_\mathcal{M}(\mathbb{R})$, then $[\pi_a^*,M_b]$ is bounded on $L_p(L_\infty(\mathbb{R})\otimes \mathcal{M})$ 
	and
	\begin{equation*}
		\|[\pi_a^*,M_b]\|_{L_p(L_\infty(\mathbb{R})\otimes \mathcal{M})\to L_p(L_\infty(\mathbb{R})\otimes \mathcal{M})}\lesssim_{d,p} \|a\|_{BMO^d(\mathbb{R})}\|b\|_{BMO^d_{\mathcal{M}}(\mathbb{R})}.
	\end{equation*}
This is because of Theorem \ref{3} and
\begin{equation*}
[\pi_a^*,M_b]^*=-[\pi_a,M_{b^*}].
\end{equation*}
\end{rem}


\bigskip {\textbf{Acknowledgments.}} We thank Professor Quanhua Xu for his careful reading of different versions of this paper. We are also very grateful to Professor Guixiang Hong for helpful suggestions.

This work is partially supported by the French ANR project (No. ANR-19-CE40-0002).

\bibliographystyle{myrefstyle}
\bibliography{ref666}

\end{document}